\documentclass[a4paper]{article}
\usepackage{amsmath,amsthm,amsfonts}
\usepackage{amssymb}
\usepackage[T1]{fontenc}

\newcommand{\ip}[2]{\bigl\langle #1,\, #2\bigr\rangle}	

\usepackage{graphicx}
\usepackage{subfigure}
\usepackage{euscript}
\usepackage{units}
\usepackage{enumitem}
\usepackage[bookmarks=true,hidelinks]{hyperref}
\usepackage{cleveref}
\usepackage{bbm}
\usepackage{bm}
\usepackage[table,xcdraw]{xcolor}
\usepackage{xcolor}
\usepackage[utf8]{inputenc}
\usepackage{epstopdf}

\newtheorem{theorem}{Theorem}
\newtheorem{lemma}[theorem]{Lemma}
\newtheorem{proposition}[theorem]{Proposition}

\theoremstyle{definition}
\newtheorem{definition}[theorem]{Definition}
\newtheorem{remark}[theorem]{Remark}

\numberwithin{equation}{section}

\def\XXint#1#2#3{{\setbox0=\hbox{$#1{#2#3}{\int}$ }
\vcenter{\hbox{$#2#3$ }}\kern-.56\wd0}}

\renewcommand{\geq}{\geqslant}
\renewcommand{\leq}{\leqslant}

\renewcommand{\epsilon}{\varepsilon}
\newcommand{\eps} {\varepsilon}
\renewcommand{\phi}{\varphi}
\newcommand{\R}{\mathbb{R}}
\newcommand{\T}{\mathbb{T}}

\newcommand{\Z}{\mathbb{Z}}
\newcommand{\K}{\mathbb{K}}
\newcommand{\Meas}{{\EuScript{M}}}
\newcommand{\Prob}{{\EuScript{P}}}
\newcommand{\Lip}{{\textrm{Lip}}}
\newcommand{\TV}{{\textrm{TV}}}
\newcommand{\BV}{{\textrm{BV}}}

\DeclareMathOperator{\supp}{supp}

\newcommand{\ind}{\mathbbm{1}} 

\newcommand{\Dt}{{\Delta t}}
\newcommand{\Dx}{{\Delta x}}
\newcommand{\Dy}{{\Delta y}}
\newcommand{\xset}{{\EuScript{K}}}	
\newcommand{\iset}{{\EuScript{I}}}	
\newcommand{\hf}{{\unitfrac{1}{2}}}
\newcommand{\thf}{{\unitfrac{3}{2}}}
\newcommand{\imhf}{{i-\hf}}
\newcommand{\iphf}{{i+\hf}}
\newcommand{\jmhf}{{j-\hf}}
\newcommand{\jphf}{{j+\hf}}
\newcommand{\ipthf}{{i+\thf}}
\newcommand{\dd}{\,d}

\begin{document}
\title{A convergent finite volume method for the Kuramoto equation and related non-local conservation laws}
\author{N. Chatterjee\thanks{Department of Mathematics, University of Oslo, Postboks 1053 Blindern, 0316 Oslo, Norway.} \and U. S. Fjordholm$^*$}
\maketitle

\begin{abstract}
We derive and study a Lax--Friedrichs type finite volume method for a large class of nonlocal continuity equations in multiple dimensions. We prove that the method converges weakly to the measure-valued solution, and converges strongly if the initial data is of bounded variation. Several numerical examples for the kinetic Kuramoto equation are provided, demonstrating that the method works well both for regular and singular data.
\end{abstract}

\section{Introduction}
The purpose of the present paper is to derive a non-oscillatory and convergent numerical method for a large class of nonlocal continuity equations of the form
\begin{equation}\label{eq:conslaw}
\begin{split}
\partial_t \mu + \nabla\cdot\big(V[\mu]\mu\big) = 0 &\qquad x\in\xset,\ t>0 \\
\mu(x,0) = \mu_0(x) & \qquad x\in\xset
\end{split}
\end{equation}
on some domain $\xset \subset \R^d$. Here, $\mu(x,t)$ is interpreted as a density of particles at time $t$, and $V[\mu]:\xset\to\R^d$ depends nonlocally on $\mu$, typically the convolution of $\mu$ with some kernel. Conservation laws of this form arise in a large variety of applications, such as the simulation of crowd dynamics, microbiology, flocking of birds, traffic flow and more \cite{Buck,Kuramoto,Winfree}. One particular instance of \eqref{eq:conslaw} which has been studied extensively is the so-called \emph{Kuramoto--Sakaguchi equation}, also called the \emph{kinetic Kuramoto equation},
\begin{equation}\label{eq:ks}
\begin{split}
\partial_{t}f + \partial_{\theta}(\omega[f]f) = 0 &\qquad (\theta,\Omega)\in\T\times\R,\ t>0 \\
f(\theta, \Omega, 0) = f_{0}(\theta, \Omega) &\qquad (\theta,\Omega)\in\T\times\R
\end{split}
\end{equation}
where $\T = \R/2\pi$ is the one-dimensional torus and
\[
\omega[f](\theta, \Omega, t):=\Omega - K\int_{\T}\mbox{sin}(\theta-\theta^*)\rho(\theta^*,t) \dd\theta^*, \qquad \rho(\theta,t) := \int_\R f(\theta, \Omega, t)\, d\Omega
\]
This equation arises as the mean-field limit of the \emph{Kuramoto equation}, a system of ordinary differential equations for coupled oscillators which was first studied by Winfree and Kuramoto \cite{KuramotoI,Kuramoto,Winfree}. The unknowns in the Kuramoto equation are the phase $\theta\in\T$ and the natural frequency $\Omega\in\R$ of each oscillator, and the interaction between the oscillators depends on the coupling strength $K>0$ and the relative difference in phase $\theta-\theta^*$ between pairs of oscillators. The mean-field limit as the number of oscillators goes to infinity is a probability distribution $f=f(\theta,\Omega,t)$ obeying the above nonlocal continuity equation; see e.g.~\cite{Dobrushin1979,Neunzert,Chiba}. See also the recent paper \cite{Carrillo} for some qualitative properties of \eqref{eq:ks}.

We will also let $g(\Omega,t):=\int_\T f(\theta,\Omega,t)\,d\theta$ denote the distribution function for natural frequencies. From \eqref{eq:ks} it is easily seen that $g$ is constant in time, $g(\Omega,t)\equiv g(\Omega)$. Here and elsewhere we will assume that $f_0$ (and hence also $g$) is compactly supported.

The above equations are valid for oscillators with non-identical natural frequencies. The situation is somewhat simpler in the case of identical oscillators, i.e.\ oscillators whose natural frequencies coincide. This corresponds to $g$ being a Dirac measure, and without loss of generality we can assume that $g=\delta$, the Dirac measure located at $\Omega = 0$. Consequently,
\[
f(\theta, \Omega, t) = \rho(\theta, t) \delta(\Omega), \qquad \omega[f]f = L[\rho]\rho, \qquad L[\rho](\theta,t):=-K\int_{\T}\mbox{sin}(\theta - \theta^*) \rho(\theta^*,t) \dd\theta^*.
\]
Therefore, \eqref{eq:ks} reduces to the following equation for $\rho$:
\begin{equation}\label{eq:ksident}
\begin{split}
\partial_t\rho + \partial_\theta(L[\rho]\rho) = 0 \\
\rho(\theta, 0) = \rho_{0}(\theta).
\end{split}
\end{equation}
Both equation \eqref{eq:ks} and \eqref{eq:ksident} are instances of general nonlocal conservation laws of the form \eqref{eq:conslaw}.

The purpose of the present paper is to derive and analyze a finite volume numerical method for a large class of equations of the form \eqref{eq:conslaw} (including the kinetic Kuramoto equations \eqref{eq:ks}, \eqref{eq:ksident}). In particular, we prove that the scheme converges strongly to the unique weak solution $\mu$ whenever $\mu_0\in L^1\cap \textrm{BV}$, and in all other cases converges in the sense of measures to the so-called \emph{measure-valued} solution $\mu$. We emphasize that the stability and convergence properties of the scheme are valid even when $\mu_0$ (and hence also the exact solution) has point-mass singularities.

In the recent works \cite{Amadori,AmadoriI}, Amadori et al.~designed and analyzed a front-tracking numerical method for the the Kuramoto--Sakaguchi equation for identical oscillators \eqref{eq:ksident}. In contrast to their method, our finite volume method works in any number of dimensions, does not require any regularity on the initial data, and does not impose ``entropy conditions'' on solutions. We also mention the work by Crippa and L\'ecureux-Mercier \cite{CriMer13} where well-posedness of a \emph{system} of nonlocal continuity equations of the form \eqref{eq:conslaw} is established. The extension of our finite volume scheme to such systems should be straightforward.

\subsection{Nonlocal conservation laws}
Nonlocal conservation laws of the form \eqref{eq:conslaw} were first studied by Neunzert \cite{Neunzert} and Dobrushin \cite{Dobrushin1979}. Using techniques which by now are standard, they showed existence and uniqueness of solutions of \eqref{eq:conslaw}. We will consider solutions which are weakly continuous maps $\mu : \R_+ \to \Meas(\xset)$, mapping time $t$ to probability measures $\mu_t\in \Prob(\xset)$. Here, $\Meas(\xset) = (C_0(\xset))^*$ is the space of bounded Radon measures on $\xset$, $\Prob(\xset)$ is the subset of probability measures, and by weakly continuous we mean that $t \mapsto \ip{\mu_t}{\phi} := \int_\xset \phi(x)\,d\mu_t(x)$ is continuous for every $\phi\in C_0(\xset)$. We define the \emph{1-Wasserstein} metric
\[
d_W(\nu_1,\nu_2):=\sup\left\{\int_\xset \psi(x)\,d(\nu_2-\nu_1)(x)\,:\, \psi\in C_b(\xset),\ \|\psi\|_{\Lip(\xset)}\leq 1 \right\}, \qquad \nu_1,\nu_2\in\Prob(\xset).
\]
It can be shown that $d_W$ metrices the topology of weak (or \emph{narrow}) convergence in $\Prob_1(\xset)$, the set of probability measures $\nu$ with finite first moment, $\int_\xset |x|\,d\nu(x)<\infty$ (see e.g.~\cite{Villani}).

\begin{definition}
$\mu \in C_w(\R_+; \Meas(\xset))$ is a \emph{measure-valued solution} of \eqref{eq:conslaw} if it satisfies \eqref{eq:conslaw} in the sense of distributions, i.e.
\begin{equation}\label{eq:measvalsoln}
\int_0^\infty\int_\xset \partial_t\phi(x,t) + V[\mu_t](x)\cdot\nabla\phi(x,t)\,d\mu_t(x)\,dt + \int_\xset \phi(x,0)\,d\mu_0(x) = 0
\end{equation}
for every $\phi\in C_c^\infty(\xset\times\R_+)$.
\end{definition}

Since $\mu$ is assumed to be weakly continuous in time, one can show that $\mu$ is a measure valued solution \emph{if and only if}
\begin{equation}\label{eq:weaksolndef2}
\ip{\mu_t}{\psi}-\ip{\mu_s}{\psi} = \int_{s}^{t}\int_\xset V[\mu_\tau](x)\cdot\nabla\psi(x)\,d\mu_\tau(x)\,d\tau
\end{equation}
for every $0\leq s<t$ and every $\psi\in C_c^\infty(\xset)$. We say that $\mu$ is a \emph{weak solution} of \eqref{eq:conslaw} if it is a measure-valued solution which is absolutely continuous with respect to Lebesgue measure.

We will henceforth assume that $\xset$ is of the form $\xset = \K_1\times\cdots\times\K_d$, where each $\K_i$ is either the torus $\T$ or the whole real line $\R$. For the well-posedness of \eqref{eq:conslaw} and the convergence of our numerical scheme, we also need some of the following assumptions on $V$:
\begin{enumerate}[label=(A\arabic*)]
\item\label{cond:linfbound} $\exists\ C_1>0$ such that $\|V[\nu]\|_{L^\infty(\xset)} \leq C_1$ for all $\nu\in\Prob(\xset)$
\item\label{cond:spacelipschitz} $\exists\ C_2>0$ such that $|V[\nu](x)-V[\nu](y)| \leq C_2|x-y|$ for all $x,y\in\xset$ and $\nu\in\Prob(\xset)$
\item\label{cond:lipschitz} $\exists\ C_3>0$ such that $\|V[\nu_1]-V[\nu_2]\|_{L^\infty(\xset)} \leq C_3d_W(\nu_1,\nu_2)$ for all $\nu_1,\nu_2\in\Prob(\xset)$
\item\label{cond:c2bound} $\exists\ C_4>0$ such that $\big\|\nabla V[\nu]\big\|_{\Lip(\xset)} \leq C_4$ for all $\nu\in\Prob(\xset)$.
\end{enumerate}

The main well-posedness result for the nonlocal conservation law is the following.
\begin{theorem}[Neunzert \cite{Neunzert}, Dobrushin \cite{Dobrushin1979}]\label{thm:conslaw}
Let $\mu_0\in\Prob_1(\xset)$ and let $V$ satisfy assumptions \ref{cond:linfbound}, \ref{cond:spacelipschitz}, \ref{cond:lipschitz}. Then there exists a unique measure-valued solution of \eqref{eq:conslaw}. This solution is Lipschitz in time, $d_W(\mu_t,\mu_s)\leq C_1|t-s|$. If $\mu_0 \in L^1(\xset)$ then $\mu_t\in L^1(\xset)$ for all $t>0$.
\end{theorem}

\begin{remark}
It is easy to see that both kinetic Kuramoto equations \eqref{eq:ks} and \eqref{eq:ksident} satisfy assumptions \ref{cond:linfbound}--\ref{cond:c2bound}, and are therefore well-posed by Theorem \ref{thm:conslaw}. Indeed, in the case of \eqref{eq:ks} we let $d=2$, $\xset = \T\times\R$ and $V[\nu] = \big(V^1[\nu],V^2[\nu]\big)$, where $V^2[\nu]\equiv 0$ and
\[
V^1[\nu](\theta,\Omega) = \begin{cases}
\Omega - K\int_\xset \sin(\theta-\theta^*)\,d\nu(\theta^*,\Omega^*) & \text{for } \Omega\in\supp g \\
0 & \text{otherwise.}
\end{cases}
\]
We can then choose $C_1 = K + \max\{|\Omega| : \Omega\in\supp g\}$, $C_2=1+K$ and $C_3=C_4=K$. In the case of \eqref{eq:ksident} we let $d=1$, $\xset=\T$, $V=L$ and $C_1=C_2=C_3=C_4=K$.
\end{remark}

``Kinetic'' PDEs \eqref{eq:conslaw} under our assumptions (A1)--(A3) are rather well-behaved under approximations. For instance, in a \emph{particle approximation} one approximates the solution as a convex combination of Dirac measures, $\mu_t^M = \sum_{i=1}^M \delta_{x_i(t)}\mu_i$, where $x_i(t)$ is the position of the $i$th particle and $\mu_i\in(0,\infty)$ is its mass. It is straightforward to see that $\mu^M$ is in fact a measure-valued solution of \eqref{eq:conslaw} provided $x_i(t)$ satisfy the system of ODEs
\begin{equation}\label{eq:particleapprox}
\frac{dx_i}{dt}(t) = V[\mu^M](x_i(t)).
\end{equation}
Assumptions \ref{cond:linfbound}, \ref{cond:spacelipschitz} guarantee that this system has a unique solution. Taking an approximating sequence $\mu^M_0$ converging weakly to $\mu_0$, assumption \ref{cond:lipschitz} guarantees that the limit is a measure-valued solution. Thus, if one can solve the system of ODEs \eqref{eq:particleapprox} then the question of convergence boils down to the approximation of the initial data by Dirac measures \cite{Neunzert,Dobrushin1979}. 

Although the dynamical system \eqref{eq:particleapprox} can be studied qualitatively in certain simple cases, it is computationally infeasible to solve \eqref{eq:particleapprox} in more realistic settings. In the next section we proceed to construct a simple, computationally efficient numerical approximation to the nonlocal PDE \eqref{eq:conslaw}.

\section{The Lax--Friedrichs scheme}

\subsection{Derivation of the method}
For the sake of simplicity we derive the method in one space dimension with either $\xset=\R$ or $\xset=\T$, and then simply state the method in multiple space dimensions (Section \ref{sec:multid}). We start by deriving a \emph{staggered} version of the method but---again for the sake of simplicity---we will only analyze the unstaggered version of this method.

Consider a mesh $\ldots < x_\imhf < x_\iphf < \ldots$, where $i$ run over $\iset:=\Z$ in the unbounded case $\xset=\R$, or over some finite set $\iset:=\{1,\dots,N_x\}$ in the periodic case $\xset=\T$, and such that $\bigcup_i [x_\imhf,x_\iphf)=\xset$. For simplicity we assume a uniform mesh, $x_\iphf-x_\imhf\equiv \Dx$. We will denote $x_i:=\frac{x_\imhf+x_\iphf}{2}$. Let $0=t^0<t^1<\dots<t^N=T$ be a partition of the time interval $[0,T]$ with uniform step size $t^{n+1}-t^n\equiv\Dt$. Assuming that we are given an approximate solution $\mu^n=\sum_i \delta_{x_i}\mu_i^n$ at time $t^n$, we compute an approximation $\mu^{n+1}$ at $t=t^{n+1}$ as follows. Let $\mu_t$ be the exact solution of 
\begin{equation}\label{eq:godunov}
\begin{split}
\partial_t \mu + \partial_x\big(V[\mu]\mu\big) = 0  &\qquad t^n<t<t^{n+1}\\
\mu_{t^n} = \mu^n.
\end{split}
\end{equation}
Define $\mu_\iphf^{n+1}:= \ip{\mu_{t^{n+1}}}{\psi_\iphf}$, where $\psi$ are the usual ``witch's hat'' finite element basis functions,
\[
\psi_\iphf(x) := \begin{cases}
\frac{x-x_\imhf}{\Dx} & x_\imhf\leq x<x_\iphf \\
\frac{x_\ipthf-x}{\Dx} & x_\iphf\leq x<x_\ipthf \\
0 & \text{else,}
\end{cases}
\]
and the index $i$ is taken over all $i\in\iset$ if $n$ is even and over $i\in\iset-\hf$ is $n$ is odd. The approximation $\mu^{n+1}$ at time $t^{n+1}$ is then defined to be the projected solution $\mu^{n+1} := \sum_i \mu_\iphf^{n+1}\delta_{x_\iphf}$. 

We can derive a simplified expression of $\mu_\iphf^{n+1}$ as follows. Since the initial data $\mu^n$ in \eqref{eq:godunov} is a convex combination of Dirac measures, the solution of \eqref{eq:godunov} can be written as $\mu_t = \sum_i \delta_{x_i(t)}\mu_i^n$, where $x_i$ solve the system of ODEs \eqref{eq:particleapprox}. If we assume the \emph{CFL condition}
\begin{equation}\label{eq:CFLstagg}
\frac{\Dt}{\Dx}\big\|V[\mu]\big\|_{L^\infty(\xset)} \leq \frac{1}{2}
\end{equation}
then the particle with position $x_i(t)$ stays in the interval $(x_\imhf,x_\iphf)$ for all $t\in[t^n,t^{n+1})$. In particular, the atoms $\delta_{x_i(t)}$ in $\mu_t$ stay away from the kinks in $\psi_\iphf$, and hence we may use the (non-differentiable) function $\psi_\iphf$ as a test function in the weak formulation \eqref{eq:weaksolndef2}. Using the fact that $\ip{\mu^n}{\psi_\iphf} = \frac{\mu_i^n+\mu_{i+1}^n}{2}$, we obtain
\begin{align*}
\mu_\iphf^{n+1} &= \frac{\mu_i^n+\mu_{i+1}^n}{2} + \int_{t^n}^{t^{n+1}}\int_\R V[\mu_t](x)\partial_x\psi(x)\,d\mu_t(x)dt \\
&= \frac{\mu_i^n+\mu_{i+1}^n}{2} - \frac{1}{\Dx}\int_{t^n}^{t^{n+1}}V[\mu_t](x_{i+1}(t))\mu_{i+1}^n - V[\mu_t](x_{i}(t))\mu_{i}^n\,dt.
\end{align*}
Approximating $V[\mu_t](x_{i}(t)) \approx V[\mu^n](x_{i})$ yields our final scheme,
\begin{equation}\label{eq:stagglaxfr}
\mu_\iphf^{n+1} = \dfrac{\mu_i^n+\mu_{i+1}^n}{2} - \Dt\dfrac{V[\mu^n]_{i+1}\mu_{i+1}^n - V[\mu^n]_i\mu_{i}^n}{\Dx}
\end{equation}
(where we denote $V[\mu^n]_i := V[\mu^n](x_i)$). The initial data is set as $\mu_i^0 = \ip{\mu_0}{\psi_i}$. We refer to \eqref{eq:stagglaxfr} as the \emph{staggered Lax--Friedrichs method}, after \cite{Tadmor}.

For the sake of simplicity, we will only consider an unstaggered version of \eqref{eq:stagglaxfr}, obtained by inserting an extra mesh point between all pairs of neighboring mesh points. The \emph{unstaggered Lax--Friedrichs method} is then
\begin{equation}\label{eq:unstagglaxfr}
\begin{split}
\mu_i^0 &= \ip{\mu_0}{\psi_i} \\
\mu_i^{n+1} &= \dfrac{\mu_{i-1}^n+\mu_{i+1}^n}{2} - \Dt\dfrac{V[\mu^n]_{i+1}\mu_{i+1}^n - V[\mu^n]_{i-1}\mu_{i-1}^n}{2\Dx}
\end{split}
\end{equation}
where $\mu^n = \sum_i \mu_i^n\delta_{x_i}$.

\subsection{Properties of the method}
In this section we prove several stability properties of the staggered Lax--Friedrichs method. Since the multi-dimensional method shares the same properties as the one-dimensional method, we will prove these properties for the one-dimensional method \eqref{eq:unstagglaxfr} and merely state the properties for the two-dimensional method \eqref{eq:unstagglaxfr2d} in Section \ref{sec:multid}.

\begin{proposition}\label{prop:stabprop1d}
Consider the unstaggered, one-dimensional Lax--Friedrichs method \eqref{eq:unstagglaxfr} and define the piecewise linear measure-valued map
\begin{equation}\label{eq:approxsolndef}
\mu_t^\Dx := \frac{t^{n+1}-t}{\Dt}\mu^n + \frac{t-t^n}{\Dt}\mu^{n+1} \qquad \text{for } t\in[t^n,t^{n+1})
\end{equation}
where $\mu^n := \sum_i \mu_i^n \delta_{x_i}$. Assume that $V$ satisfies condition \ref{cond:linfbound}, and that the {CFL} condition
\begin{equation}\label{eq:CFL}
\lambda_0 \leq \frac{\Dt}{\Dx} \leq C_1^{-1}
\end{equation}
is satisfied for some $\lambda_0>0$ (where $C_1$ is the constant in \ref{cond:linfbound}). Then for all $t\geq0$
\begin{enumerate}[label=(\roman*)]
\item\label{property:positive} $\mu^\Dx_t \geq 0$
\item\label{property:unitmass} $\|\mu^\Dx_t\|_{\Meas(\K)} = 1$
\item\label{property:finitepropspeed} if $\supp\mu_0\subset B_M(0)$ then $\supp\mu^\Dx_t\subset B_{M+\lambda_0^{-1}t}(0)$
\item\label{property:timecontwass} $d_W\big(\mu^\Dx_t,\mu^\Dx_s\big) \leq \lambda_0^{-1}|t-s|$ for all $s\geq0$.
\end{enumerate}
\end{proposition}
\begin{proof}
Writing
\[
\mu_i^{n+1} = \mu_{i-1}^n\frac{1+\frac{\Dt}{\Dx}V[\mu^n]_{i-1}}{2} + \mu_{i+1}^n\frac{1-\frac{\Dt}{\Dx}V[\mu^n]_{i+1}}{2},
\]
it is clear that \eqref{eq:CFL} and \ref{cond:linfbound} ensure that $\mu_i^{n+1}$ is a convex combination of $\mu_{i-1}^n$ and $\mu_{i+1}^n$, whence \ref{property:positive} and \ref{property:unitmass} follow. From \eqref{eq:unstagglaxfr} we see that if $\mu_{i-1}^n=\mu_i^n=\mu_{i+1}^n=0$ then $\mu_i^{n+1}=0$. Hence, after $n$ time steps the support of $\mu_t^\Dx$ can have grown at most a distance $n\Dx \leq n\frac{\Dt}{\lambda_0} = \lambda_0^{-1}t^n$, and \ref{property:finitepropspeed} follows. For \ref{property:timecontwass}, it is clear that by the definition \eqref{eq:approxsolndef} of $\mu^\Dx$, it is enough to prove the claim for $s=t^n$, $t=t^{n+1}$. Let $\psi\in C_b(\K)$ be Lipschitz continuous and write $\psi_i=\psi(x_i)$. Then
\begin{align*}
\ip{\mu^{n+1}-\mu^n}{\psi} &= \sum_i \psi_i\left(\frac{\mu_{i-1}^n-2\mu_i^n+\mu_{i+1}^n}{2} - \frac{\Dt}{2\Dx}\big(V[\mu^n]_{i+1}\mu_{i+1}^n-V[\mu^n]_{i-1}\mu_{i-1}^n\big)\right) \\
\intertext{\textit{(summation by parts)}}
&= -\frac{1}{2}\sum_i (\psi_{i+1}-\psi_i)\left(\mu_{i+1}^n-\mu_i^n - \frac{\Dt}{\Dx}\big(V[\mu^n]_{i}\mu_{i}^n+V[\mu^n]_{i+1}\mu_{i+1}^n\big)\right) \\
&= \frac{1}{2}\sum_i (\psi_{i+1}-\psi_i)\left(\mu_{i}^n\left(1 + \frac{\Dt}{\Dx}V[\mu^n]_i\right) + \mu_{i+1}^n\left(\frac{\Dt}{\Dx}V[\mu^n]_{i+1}-1\right)\right) \\
&\leq \frac{1}{2}\Dx\|\psi\|_\Lip\sum_i \mu_{i}^n\left(1 + \frac{\Dt}{\Dx}V[\mu^n]_i\right) + \mu_{i+1}^n\left(1-\frac{\Dt}{\Dx}V[\mu^n]_{i+1}\right) \\
&= \Dx\|\psi\|_\Lip,
\end{align*}
the last two steps following from the CFL condition and \ref{property:unitmass}. Taking the supremum over $\psi$ with $\|\psi\|_\Lip\leq1$ yields \ref{property:timecontwass}.
\end{proof}

\subsection{Weak convergence of the method}
We split the proof of convergence of the numerical  method to the (unique) measure-valued solution into two parts. First we show that our method is consistent (Lemma \ref{lem:laxwendroff}), in the sense that \emph{if} the method converges, then the limit is the measure-valued solution. Next, we show that the method indeed converges either weakly (Theorem \ref{thm:convmeasval}) or strongly (Theorem \ref{thm:convstrongsoln}), depending on the assumptions on the velocity field $V$ and the initial data $\mu_0$.

\begin{lemma}\label{lem:laxwendroff}
Assume that $V$ satisfies conditions \ref{cond:spacelipschitz}, \ref{cond:lipschitz} and define $\mu^\Dx$ by \eqref{eq:approxsolndef}. Assume that $d_W(\mu_t^\Dx,\mu_t) \to 0$ as $\Dx\to0$ uniformly on bounded intervals $t\in[0,T]$, for some $\mu \in C_w([0,\infty);\Meas(\K))$. Then $\mu$ is the measure-valued solution of \eqref{eq:conslaw}.
\end{lemma}
\begin{proof}
The convergence $d_W(\mu_t^\Dx,\mu_t) \to 0$ implies that $\mu_t\in\Prob(\xset)$ for all $t$. Let $\phi\in C_c^\infty(\R\times\R)$; we want to show that $\mu$ satisfies \eqref{eq:measvalsoln} with $\xset=\K$. Denoting $\phi_i^n=\phi(x_i,t^n)$, we multiply \eqref{eq:unstagglaxfr} by $\phi_i^n$ and sum over $i,n$:
\begin{align*}
0 &= \sum_{n=0}^\infty\sum_i \left(\phi_i^n\mu_i^{n+1}-\phi_i^n\frac{\mu_{i+1}^n+\mu_{i+1}^n}{2} + \Dt\phi_i^n\frac{V[\mu^n]_{i+1}\mu_{i+1}^n-V[\mu^n]_{i-1}\mu_{i-1}^n}{2\Dx}\right) \\
\intertext{\textit{(summation by parts)}}
&= -\Dt\sum_{n=0}^\infty\sum_i \frac{\frac{\phi_{i-1}^n+\phi_{i+1}^n}{2}-\phi_i^{n-1}}{\Dt}\mu_i^{n} - \Dt\sum_{n=0}^\infty\sum_i \frac{\phi_{i+1}^n-\phi_{i-1}^n}{2\Dx}V[\mu^n]_i\mu_i^n - \sum_i\phi_i^{-1}\mu_i^0 \\
&= -\int_0^\infty\int_{\K} \frac{\frac{\phi(x-\Dx,t)+\phi(x+\Dx,t)}{2}-\phi(x,t)}{\Dt} + \frac{\phi(x+\Dx,t)-\phi(x-\Dx,t)}{2\Dx}V[\mu^\Dx_t](x)\,d\mu_t^\Dx(x)dt \\
&\quad - \int_{\K}\phi(x,-\Dt)\,d\mu_0^\Dx(x). 
\end{align*}
By condition \ref{cond:lipschitz} we know that $V[\mu_t^\Dx] \to V[\mu_t]$ uniformly on $\R\times[0,T]$, and by condition \ref{cond:spacelipschitz}, the functions $V[\mu_t^\Dx]$ are uniformly Lipschitz. It follows that the above integral converges to
\[
-\int_0^\infty\int_{\K} \partial_t\phi(x,t) + \partial_x\phi(x,t)V[\mu_t](x)\,d\mu_t(x)dt - \int_{\K}\phi(x,0)\,d\mu_0(x)
\]
and the proof is complete.
\end{proof}

For the convergence proof we use the following compactness lemma, whose proof is postponed to the appendix.
\begin{lemma}\label{lem:compactness}
Let $K\subset\Prob_1(\R^d)$ be a bounded set, i.e.\ $\sup_{\mu\in K}d_W(\mu,\bar{\mu})<\infty$ for some $\bar{\mu}\in\Prob(\R^d)$. Then $K$ is relatively compact in the metric space $(\Prob_1(\R^d),d_W)$.
\end{lemma}

\begin{theorem}\label{thm:convmeasval}
Assume that $V$ satisfies conditions \ref{cond:linfbound}, \ref{cond:spacelipschitz}, \ref{cond:lipschitz} and let $\mu_0\in\Prob(\K)$ have compact support. Assume that the CFL condition \eqref{eq:CFL} is satisfied. Then the Lax--Friedrichs method \eqref{eq:unstagglaxfr} converges weakly to the measure-valued solution of \eqref{eq:conslaw}. More precisely, if $\mu^\Dx$ if given by \eqref{eq:approxsolndef} then
\[
\sup_{t\in[0,T]}d_W\big(\mu^\Dx_t,\mu_t\big) \to 0 \qquad \text{as } \Dx \to 0
\]
for any $T>0$.
\end{theorem}
\begin{proof}
By Proposition \ref{prop:stabprop1d} \ref{property:positive}, \ref{property:unitmass} and \ref{property:finitepropspeed}, the measures $\mu^\Dx_t$ stay in $\Prob(\R)$ and have uniformly bounded support for all $\Dx>0$, $t\in[0,T]$, and by \ref{property:timecontwass}, the maps $\mu^\Dx:[0,T]\to\Prob(\R^d)$ are Lipschitz. Moreover, $d_W(\mu_0,\mu^\Dx_t) \leq d_W(\mu_0,\mu^\Dx_0)+d_W(\mu^\Dx_0,\mu^\Dx_t) \leq \Dx + \frac{1}{\lambda_0}t$, so by Lemma \ref{lem:compactness} the set
\[
K := \big\{\mu^\Dx_t\ :\ 0<\Dx<1,\ t\in[0,T] \big\}
\]
is relatively compact with respect to $d_W$-convergence. Hence, by Ascoli's theorem, there exists a subsequence $\Dx_k\to0$ as $k\to\infty$ and some Lipschitz map $\mu:[0,T]\to\Prob(\R^d)$ such that $d_W(\mu^{\Dx_k}_t,\mu_t) \to 0$ as $k\to\infty$, uniformly in $t\in[0,T]$. By Lemma \ref{lem:laxwendroff}, the limit $\mu$ is the measure-valued solution of \eqref{eq:conslaw}. But this solution is unique, which implies that the whole sequence $\mu^\Dx$ converges to $\mu$.
\end{proof}

\subsection{Strong convergence of the method}
If the initial data and the velocity field are sufficiently smooth then we can show that the numerical method in fact converges \emph{strongly}. We assume that the initial data is absolutely continuous with respect to Lebesgue measure, and hence has a density function $\mu_0(x)$. This data is sampled by its cell averages,
\[
\hat{\mu}^0_i := \frac{1}{\Dx}\int_{x_\imhf}^{x_\iphf}\mu_0(x)\,dx
\]
and the numerical solution is realized as the linear-in-time, piecewise constant $L^1$ function
\begin{equation}\label{eq:strongapproxsolndef}
\hat{\mu}_t^\Dx(x) = \frac{t^{n+1}-t}{\Dt}\hat{\mu}^n + \frac{t-t^n}{\Dt}\hat{\mu}^{n+1} \qquad \text{for } t\in[t^n,t^{n+1}),
\end{equation}
where
\[
\hat{\mu}^n(x) = \sum_i \mu_i^n \ind_{[x_\imhf,x_\iphf)}(x).
\]

\begin{proposition}\label{prop:strongstabprop1d}
Consider the unstaggered, one-dimensional Lax--Friedrichs method \eqref{eq:unstagglaxfr}. Assume that $V$ satisfies conditions \ref{cond:linfbound}, \ref{cond:spacelipschitz}, \ref{cond:c2bound}, and that the {CFL} condition \eqref{eq:CFL} is satisfied. Then for all $t\in[0,T]$
\begin{enumerate}[label=(\roman*)]
\item\label{property:positive2} $\hat{\mu}^\Dx_t \geq 0$
\item\label{property:unitmass2} $\|\hat{\mu}^\Dx_t\|_{L^1(\K)} = 1$
\item\label{property:tvb} $\TV(\hat{\mu}^\Dx_t) \leq \TV(\hat{\mu}_0)e^{C_2t} + (e^{C_2t}-1)\frac{C_4}{2C_2}$
\item\label{property:timecontmeas} $\|\hat{\mu}^\Dx_t-\hat{\mu}^\Dx_s\|_{L^1(\K)} \leq C_T|t-s|$ for all $s\geq0$. 
\end{enumerate}
(In \ref{property:timecontmeas}, $C_T = C_2+\lambda_0^{-1}c_T$, where $c_T$ is the upper bound from \ref{property:tvb}, $c_T := \TV(\mu_0)e^{C_2T} + (e^{C_2T}-1)\frac{C_4}{2C_2}$.)
\end{proposition}
\begin{proof}
Properties \ref{property:positive2} and \ref{property:unitmass2} obviously follow from Proposition \ref{prop:stabprop1d} (i) and (ii). To show \ref{property:tvb} we split $\TV(\hat{\mu}^{n+1})$ as follows:
\begin{align*}
\TV(\hat{\mu}^{n+1}) &\leq \frac{1}{2}\sum_i \biggl|\big(\mu^{n}_{i+2}-\mu^{n}_{i+1}\big)\left(1-\frac{\Dt}{\Dx}V[\hat{\mu}^n]_{i+2}\right) + \big(\mu^{n}_{i}-\mu^{n}_{i-1}\big)\left(1+\frac{\Dt}{\Dx}V[\hat{\mu}^n]_{i}\right) \\
&\qquad\qquad\quad - \frac{\Dt}{\Dx}\mu^{n}_{i+1}\big(V[\hat{\mu}^n]_{i+2}-V[\hat{\mu}^n]_{i+1}\big) + \frac{\Dt}{\Dx}\mu^{n}_i\big(V[\hat{\mu}^n]_i-V[\hat{\mu}^n]_{i-1}\big)\biggr| \\
&\leq A^{n}+B^{n}
\end{align*}
where
\[
A^{n} = \frac{1}{2}\sum_i \big|\mu^{n}_{i+2}-\mu^{n}_{i+1}\big|\left(1-\frac{\Dt}{\Dx}V[\hat{\mu}^n]_{i+2}\right) + \big|\mu^{n}_{i}-\mu^{n}_{i-1}\big|\left(1+\frac{\Dt}{\Dx}V[\hat{\mu}^n]_{i}\right) \leq \TV(\hat{\mu}^n)
\]
and
\begin{align*}
B^{n} &= \frac{\Dt}{2\Dx}\sum_i \Big|\mu^n_{i+1}\big(V[\hat{\mu}^n]_{i+2}-V[\hat{\mu}^n]_{i+1}\big) - \mu^n_{i-1}\big(V[\hat{\mu}^n]_{i}-V[\hat{\mu}^n]_{i-1}\big)\Big| \\
&\leq \frac{\Dt}{2}\sum_i \Big[ |\mu^n_{i+1}-\mu^n_i| + |\mu^n_{i+1}-\mu^n_i| \Big] \frac{\big|V[\hat{\mu}^n]_{i+2}-V[\hat{\mu}^n]_{i+1}\big|}{\Dx} \\
& \quad \qquad + \frac{\Dt}{2}\sum_i \mu^n_{i-1}\frac{\big|V[\hat{\mu}^n]_{i+2}-V[\hat{\mu}^n]_{i+1}-V[\hat{\mu}^n]_{i}+V[\hat{\mu}^n]_{i-1}\big|}{\Dx^2}\Dx \\
&\leq C_2\Dt\TV(\hat{\mu}^n) + \frac{C_4}{2}\Dt.
\end{align*}
Iterating over all $n$ gives
\[
\TV(\hat{\mu}^{n+1}) \leq (1+C_2\Dt)\TV(\hat{\mu}^n) + \frac{C_4}{2}\Dt \leq \dots \leq (1+C_2\Dt)^{n+1}\left(\TV(\mu_0)+\frac{C_4}{2C_2}\right)-\frac{C_4}{2C_2}
\]
and \ref{property:tvb} follows.

Finally, we show the time continuity \ref{property:timecontmeas}:
\begin{align*}
\|\hat{\mu}^{n+1}-\hat{\mu}^n\|_{L^1(\K)} &= \sum_i|\mu^{n+1}_i-\mu^n_i|\Dx \\
&= \frac{\Dx}{2}\sum_i \bigg|(\mu_{i+1}^n-\mu_i^n)\left(1-\frac{\Dt}{\Dx}V[\hat{\mu}^n]_{i+1}\right) - (\mu_i^n-\mu_{i-1}^n)\left(1+\frac{\Dt}{\Dx}V[\hat{\mu}^n]_{i-1}\right) \\
&\qquad\qquad - \frac{\Dt}{\Dx}\mu_i^n\big(V[\hat{\mu}^n]_{i+1}-V[\hat{\mu}^n]_{i-1}\big)\bigg| \\
&\leq \frac{\Dx}{2}\sum_i |\mu_{i+1}^n-\mu_i^n|\left(1-\frac{\Dt}{\Dx}V[\hat{\mu}^n]_{i+1}\right) + |\mu_i^n-\mu_{i-1}^n|\left(1+\frac{\Dt}{\Dx}V[\hat{\mu}^n]_{i-1}\right) \\
&\qquad\qquad + \frac{\Dt}{\Dx}\mu_i^n\big|V[\hat{\mu}^n]_{i+1}-V[\hat{\mu}^n]_{i-1}\big| \\
&\leq \Dx\TV(\hat{\mu}^n) + C_2\Dt.
\end{align*}
Iterating over $n$ yields \ref{property:timecontmeas}.
\end{proof}

\begin{theorem}\label{thm:convstrongsoln}
Assume that $V$ satisfies conditions \ref{cond:linfbound}, \ref{cond:spacelipschitz}, \ref{cond:lipschitz}, \ref{cond:c2bound} and let $\mu_0\in\Prob(\K)$ be compactly supported and absolutely continuous with respect to Lebesgue measure with density $\mu_0\in\BV(\K)$. Assume that the CFL condition \eqref{eq:CFL} holds. Then the measure-valued solution $\mu$ of \eqref{eq:conslaw} is absolutely continuous, and the one-dimensional Lax--Friedrichs method \eqref{eq:unstagglaxfr} converges strongly to $\mu$. More precisely, if $\hat{\mu}^\Dx$ is given by \eqref{eq:strongapproxsolndef} then
\[
\sup_{t\in[0,T]}\big\|\hat{\mu}^\Dx_t- \mu_t\big\|_{L^1(\K)} \to 0 \qquad \text{as } \Dx \to 0
\]
for any $T>0$.
\end{theorem}
\begin{proof}
By Proposition \ref{prop:strongstabprop1d} (ii) and (iii), the set
\[
K := \big\{\hat{\mu}^\Dx_t\ :\ 0<\Dx<1,\ t\in[0,T] \big\}
\]
is uniformly bounded in $\BV(\K)$, so by Helly's theorem, $K$ is relatively compact in $L^1(\K)$. Hence, by Ascoli's theorem there is a subsequence $\Dx_k\to0$ as $k\to\infty$ and some Lipschitz map $\tilde{\mu}:[0,T]\to L^1(\K)$ such that $\hat{\mu}^{\Dx_k}_t\to\tilde{\mu}_t$ in $L^1(\K)$ as $k\to\infty$ uniformly in $t\in[0,T]$. 
Since convergence in $L^1$ implies convergence in $d_W$, also $d_W(\hat{\mu}^{\Dx_k}_t,\tilde{\mu}_t) \to 0$. 
On the other hand, Theorem \ref{thm:convmeasval} implies that $d_W(\mu^{\Dx_k}_t,\mu_t) \to 0$, where $\mu$ is the measure-valued solution of \eqref{eq:conslaw}. Noting that $d_W(\hat{\mu}^{\Dx_k}_t,\mu^{\Dx_k}_t) \leq \frac{\Dx}{2}$, we can conclude $\mu=\tilde{\mu}$. The conclusion follows.
\end{proof}

\subsection{Multiple dimensions}\label{sec:multid}
The extension to multiple space dimensions is done in a tensorial fashion; we limit ourselves to the two-dimensional variant for the sake of notational simplicity. 
As before, let $\mathbb{K}_{1}, \mathbb{K}_{2}$ be either the torus $\mathbb{T}$ or the real line $\mathbb{R}$. Let $(x_\iphf)_i$ and $(y_\jphf)_j$ be discretizations 
of $\K_1$ and $\K_2$ with mesh lengths $x_\iphf-x_\imhf\equiv\Dx$ and $y_\jphf-y_\jmhf\equiv\Dy$, respectively. With obvious notation we get the unstaggered, two-dimensional
Lax--Friedrichs method:
\begin{equation}\label{eq:unstagglaxfr2d}
\begin{split}
\mu_{i,j}^0 &= \ip{\mu_0}{\phi_{i,j}} \\
\mu_{i,j}^{n+1} &= \frac{\mu_{i-1,j}^n+\mu_{i+1,j}^n+\mu_{i,j-1}^n+\mu_{i,j+1}^n}{4} - \Dt\frac{V^1[\mu^n]_{i+1,j}\mu_{i+1,j}^n - V^1[\mu^n]_{i-1,j}\mu_{i-1,j}^n}{2\Dx} \\
&\quad  - \Dt\frac{V^2[\mu^n]_{i,j+1}\mu_{i,j+1}^n - V^2[\mu^n]_{i,j-1}\mu_{i,j-1}^n}{2\Dy}
\end{split}
\end{equation}
where we denote $\mu^n = \sum_i\sum_j \mu_{i,j}^n\delta_{(x_i,y_j)}$. As before, $\mu$ is interpolated linearly between time steps,
\begin{equation}\label{eq:approxsolndef2d}
\mu_t^{\Dx, \Dy} := \frac{t^{n+1}-t}{\Dt}\mu^n + \frac{t-t^n}{\Dt}\mu^{n+1} \qquad \text{for } t\in[t^n,t^{n+1}).
\end{equation}
We omit the proofs of the following stability and convergence results, as they are straightforward generalizations of their one-dimensional counterparts.

\begin{proposition}\label{prop:stabprop2d}
Consider the unstaggered, two-dimensional Lax--Friedrichs method \eqref{eq:unstagglaxfr2d}. Assume that $V$ satisfies condition \ref{cond:linfbound}, and that the {CFL} condition
\begin{equation}\label{eq:CFL2d}
\lambda_0\leq\frac{\Dt}{\min(\Dx,\Dy)}C_1 \leq \frac{1}{2}
\end{equation}
is satisfied for some $\lambda_0>0$. Then for all $t\geq0$
\begin{enumerate}[label=(\roman*)]
\item\label{property:positive2d} $\mu^{\Dx, \Dy}_t \geq 0$
\item\label{property:unitmass2d} $\|\mu^{\Dx, \Dy}_t\|_{\Meas(\xset)} = 1$
\item\label{property:timecontwass2d} $d_W\big(\mu^{\Dx, \Dy}_t,\mu^{\Dx, \Dy}_s\big) \leq \frac{\Dx+\Dy}{2\Dt}|t-s|$ for all $s\geq0$.
\end{enumerate}
\end{proposition}

\begin{theorem}\label{thm:convmeasval2d}
Consider the two-dimensional equation \eqref{eq:conslaw}. Assume that $V$ satisfies conditions \ref{cond:linfbound}, \ref{cond:spacelipschitz}, \ref{cond:lipschitz} and let $\mu_0\in\Prob(\K_1\times\K_2)$ have compact support. Assume that the CFL condition \eqref{eq:CFL2d} is satisfied. Then the Lax--Friedrichs method \eqref{eq:unstagglaxfr2d} converges weakly to the measure-valued solution of \eqref{eq:conslaw}. More precisely, if $\mu^{\Dx, \Dy}$ if given by \eqref{eq:approxsolndef2d} then
\[
\sup_{t\in[0,T]}d_W\big(\mu^{\Dx, \Dy}_t,\mu_t\big) \to 0 \qquad \text{as } \Dx,\Dy \to 0
\]
for any $T>0$.
\end{theorem}

As in the $1$-dimensional case, sufficient smoothness for the initial data and the velocity field will yield \textit{strong} convergence of the numerical method. Assume the initial datum is absolutely continuous with respect to the Lebesgue measure with density function $\mu_{0}(x, y)$. This initial data is sampled by its cell averages: 
\[
 \hat{\mu}^0_{i,j} := \frac{1}{\Dx \Dy}\int_{x_\imhf}^{x_\iphf} \int_{y_\jmhf}^{y_\jphf}\mu_0(x,y)\,dx dy
\]
Similarly we define the numerical solution as piecewise linear $L^1$ function as
\begin{equation}\label{eq:strongapproxsolndef2d}
\hat{\mu}_t^{\Dx, \Dy} = \frac{t^{n+1}-t}{\Dt}\hat{\mu}^n + \frac{t-t^n}{\Dt}\hat{\mu}^{n+1} \qquad \text{for } t\in[t^n,t^{n+1})
\end{equation}
where
\[
\hat{\mu}^n(x, y) = \sum_{i, j} \hat{\mu}_{i, j}^n \ind_{[x_\imhf,x_\iphf) \times [y_\jmhf,y_\jphf)}(x, y).
\]

\begin{proposition}\label{prop:strongstabprop2d}
Consider the unstaggered, one-dimensional Lax–Friedrichs method \eqref{eq:unstagglaxfr}. Assume that $V$ satisfies the conditions \ref{cond:linfbound}, \ref{cond:spacelipschitz}, \ref{cond:c2bound} and that the CFL condition \eqref{eq:CFL2d} is satisfied. Then for all $t\in[0,T]$
 
\begin{enumerate}[label=(\roman*)]
\item\label{property:positive2dd} $\hat{\mu}^{\Dx, \Dy}_t \geq 0$
\item\label{property:unitmass2dd} $\|\hat{\mu}^{\Dx, \Dy}_t\|_{L^1(\K_{1} \times \K_{2})} = 1$
\item\label{property:tvb2dd} $\TV(\hat{\mu}^{\Dx, \Dy}_t) \leq c_T:= \TV(\hat{\mu}_0)e^{C_2T} + (e^{C_2T}-1)\frac{C_4}{2C_2}$
\item\label{property:timecontmeas2dd} $\|\hat{\mu}^{\Dx, \Dy}_t-\hat{\mu}^{\Dx, \Dy}_s\|_{L^1(\K_{1} \times \K_{2})} \leq C_T|t-s|$ for all $s\geq0$. 
\end{enumerate}
(In \ref{property:timecontmeas2dd}, $C_T = C_2+\lambda_0^{-1}c_T$, where $c_T$ is the upper bound from \ref{property:tvb2dd}.)
\end{proposition}

So now we can finally state:

\begin{theorem}\label{thm:convstrongsoln2d}
Assume that $V$ satisfies conditions \ref{cond:linfbound}, \ref{cond:spacelipschitz}, \ref{cond:lipschitz}, \ref{cond:c2bound} and let $\mu_0\in\Prob(\K_1 \times \K_2)$ be compactly supported and absolutely continuous with respect to Lebesgue measure with density $\mu_0\in\BV(\K_1 \times \K_2)$. Assume that the CFL condition \eqref{eq:CFL2d} holds. Then the measure-valued solution $\mu$ of \eqref{eq:conslaw} takes values in $L^1(\K_1\times\K_2)$, and the two-dimensional Lax--Friedrichs method \eqref{eq:unstagglaxfr2d} converges strongly to $\mu$. More precisely, if $\hat{\mu}^{\Dx, \Dy}$ if given by \eqref{eq:strongapproxsolndef2d} then
\[
\sup_{t\in[0,T]}\big\|\hat{\mu}^{\Dx,\Dy}_t-\mu_t\big\|_{L^1(\K_{1} \times \K_{2})} \to 0 \qquad \text{as } \Dx, \Dy \to 0
\]
for any $T>0$.
\end{theorem}

\section{Numerical experiments}
In this section we illustrate our analytical results by performing several numerical experiments for the Kuramoto equation with identical \eqref{eq:ksident} and non-identical \eqref{eq:ks} natural frequencies. Recall that in these equations, $\theta$ takes values in the periodic domain $\T$, while $\Omega$ lies in $\R$.

\subsection{One-dimensional simulations}
We consider first the one-dimensional Kuramoto equation with identical oscillators \eqref{eq:ksident}. In all experiments we have used the unstaggered Lax--Friedrichs method \eqref{eq:unstagglaxfr} with CFL number 0.4 and with $T=0.5$ as the final time. All simulations were computed on a sequence of meshes with $N=32,64,128,256,512$ grid points, as well as a reference solution with $N=4096$ points. The tables in each subsection shows the number of grid points as well as the error and the experimental order of convergence (EOC), computed both in the 1-Wasserstein distance $d_W$ and in $L^1$.

The experiments in this section solve the Kuramoto equation with progressively more singular initial data. As will be seen from the tables, the EOC in the 1-Wasserstein distance is close to 1 in all cases, while in $L^1$ it is close to 1 only for smooth data.

\subsubsection{Polynomial initial data}\label{sec:numexp1}
The initial data is taken to be the continuous and piecewise parabolic function
\[
 \mu_{0}(\theta) = \begin{cases}
      \frac{6}{\pi^{3}}(\frac{3\pi}{2} - \theta) (\theta - \frac{\pi}{2})  & \text{if } \theta\in[\frac{\pi}{2}, \frac{3\pi}{2}) \\
      0 & \text{otherwise.}
 \end{cases}
\]
As shown in Figure \ref{fig:piecewiseparabolic}, the numerical solution is a non-oscillatory and reasonable approximation at all mesh resolutions. Table \ref{tab:piecewiseparabolic} shows that the numerical method seems to converge at a rate close to 1, both in the 1-Wasserstein and the $L^1$ distances.

\begin{figure}[h]
    \centering
    \includegraphics[width=8cm]{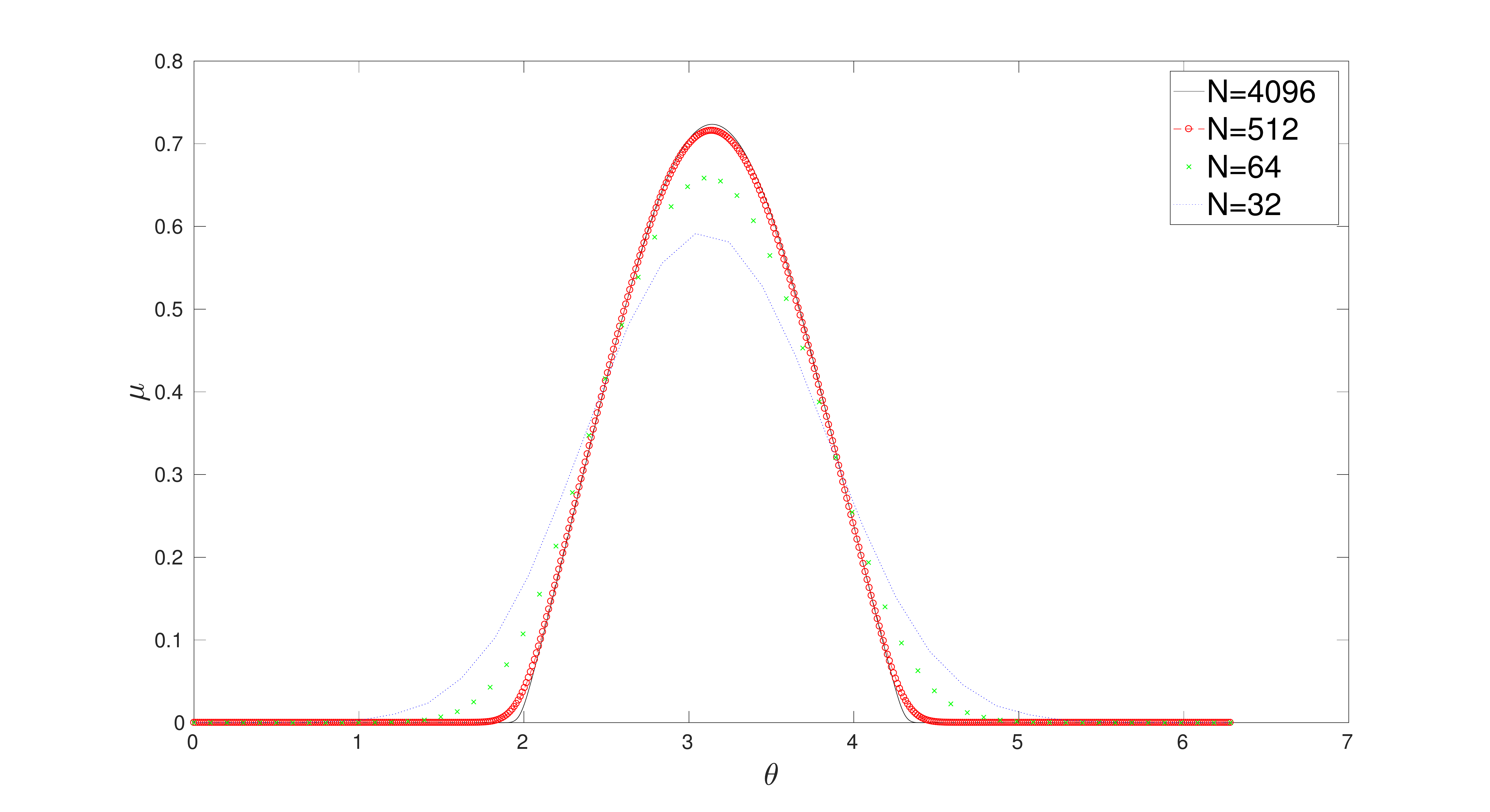}
    \caption{Comparing solutions on different meshes at $t=0.5$ for smooth initial data.} 
    \label{fig:piecewiseparabolic}
\end{figure}
\begin{table}[h]
\centering
\begin{tabular}{c|cc|cc}
& \multicolumn{2}{c|}{1-Wasserstein} & \multicolumn{2}{c}{$L^1$} \\
$N$   & Error   & EOC & Error   & EOC \\ \hline
32  & 0.1351 &  & 0.2811                                            &                                                \\
64  & 0.0634 & 1.09                  & 0.1336                                            & 1.07                                           \\
128 & 0.0303 & 1.07                  & 0.0663                                            & 1.01                                           \\
256 & 0.0153 & 0.99                  & 0.0336                                            & 0.98                                           \\
512 & 0.0073 & 1.07                  & 0.0157                                            & 1.10                                           \\
\end{tabular}
\caption{Errors and EOC in the $1$-Wasserstein and $L^1$ distances at $t=0.5$ for smooth initial data.}
\label{tab:piecewiseparabolic}
\end{table}

\subsubsection{Piecewise constant initial data}\label{sec:numexp2}
This experiment was taken from \cite{Amadori} and uses piecewise constant initial data,
\[
 \mu_{0}(\theta) =
 \begin{cases}
  \frac{2}{3\pi}  & \text{if } \theta \in [\frac{\pi}{2}, \frac{3\pi}{2}) \\
  \frac{1}{3\pi}  & \text{otherwise.}
 \end{cases}
\]
Figure \ref{fig:piecewiseconstant} shows again that the approximation is reasonably accurate even on coarse meshes. Table \ref{tab:piecewiseconstant1d} shows that the rate of convergence in 1-Wasserstein is again 1, while it is around 0.7 in $L^1$.

\begin{figure}
    \centering
    \includegraphics[width=8cm]{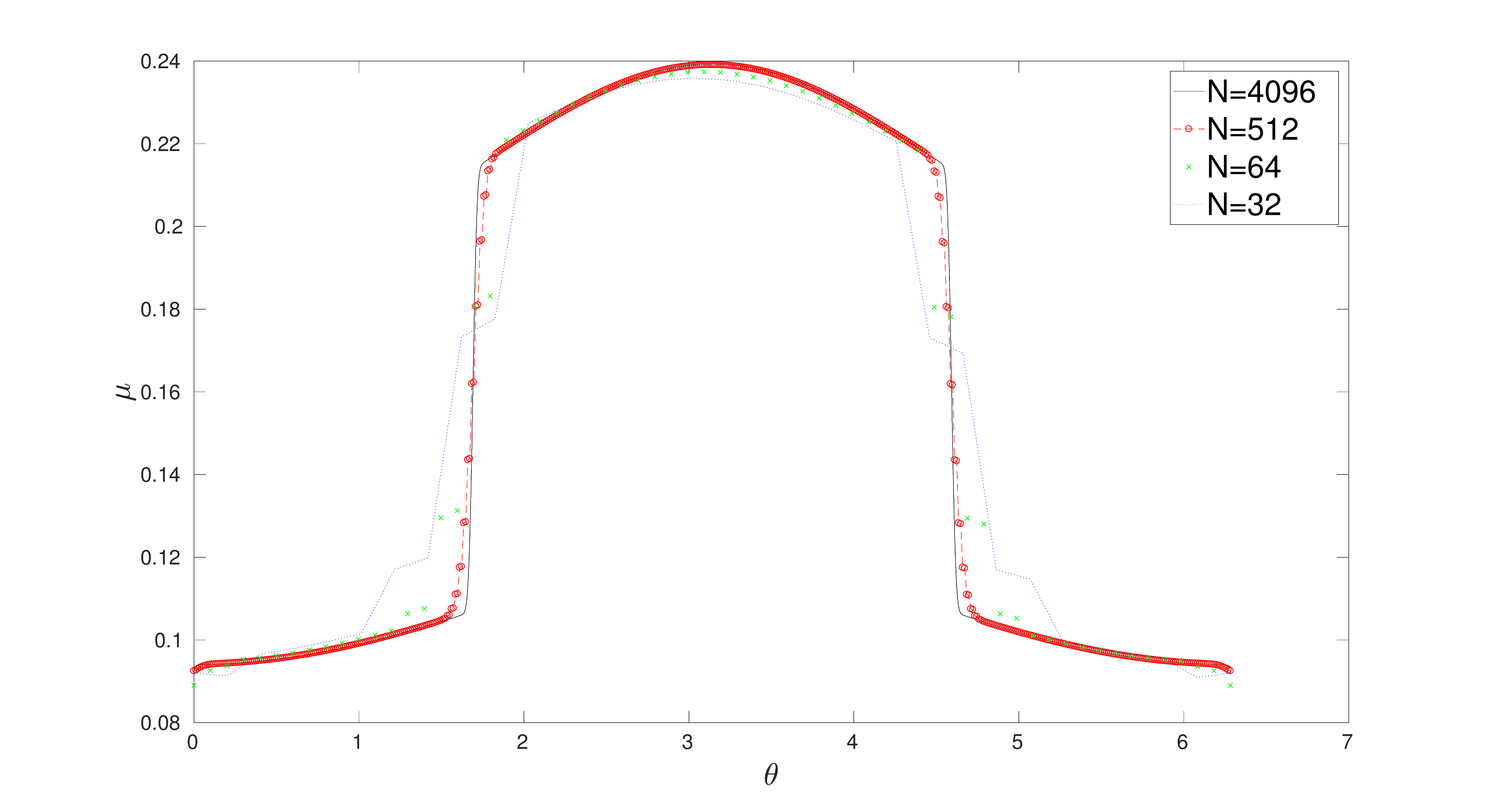}
    \caption{Comparing solutions on different meshes for piecewise constant initial data.}
    \label{fig:piecewiseconstant}
\end{figure} 
\begin{table}
\centering
\begin{tabular}{c|cc|cc}
& \multicolumn{2}{c|}{1-Wasserstein} & \multicolumn{2}{c}{$L^1$} \\
$N$   & Error   & EOC & Error   & EOC \\ \hline
32  & 0.0517 &  & 0.0610  &  \\
64  & 0.0258 & 1.00                  & 0.0315                                            & 0.95                                           \\
128 & 0.0131 & 0.98                  & 0.0196                                            & 0.69                                           \\ 
256 & 0.0067 & 0.97                  & 0.0123                                            & 0.67                                           \\ 
512 & 0.0034 & 0.98                  & 0.0071                                            & 0.79                                           \\
\end{tabular}
\caption{Errors and EOC in the $1$-Wasserstein and $L^1$ distances at $t=0.5$ for piecewise constant initial data.}
\label{tab:piecewiseconstant1d}
\end{table}

\subsubsection{Singular initial data}\label{sec:numexp3}
The final experiment uses the singular initial data
\[
 \mu_{0} = \frac{1}{4} \big(\delta_{\frac{3\pi}{4}} + \delta_{\frac{5\pi}{4}}\big) + \frac{1}{2} \chi_{[\frac{\pi}{2}, \frac{3\pi}{2}]} 
\]
where $\chi_{[a,b]}$ is the indicator function of the interval $[a,b]$. As shown in Figure \ref{fig:diracinitial}, the numerical approximation is nonoscillatory and seems to converge, even in the presence of Dirac singularities. This is confirmed in Table \ref{tab:diracinitial}: The method seems to converge at a rate of around $3/4$ in $d_W$, while it does not converge at all in $L^1$, as expected.

\begin{figure}
    \centering
    \includegraphics[width=8cm]{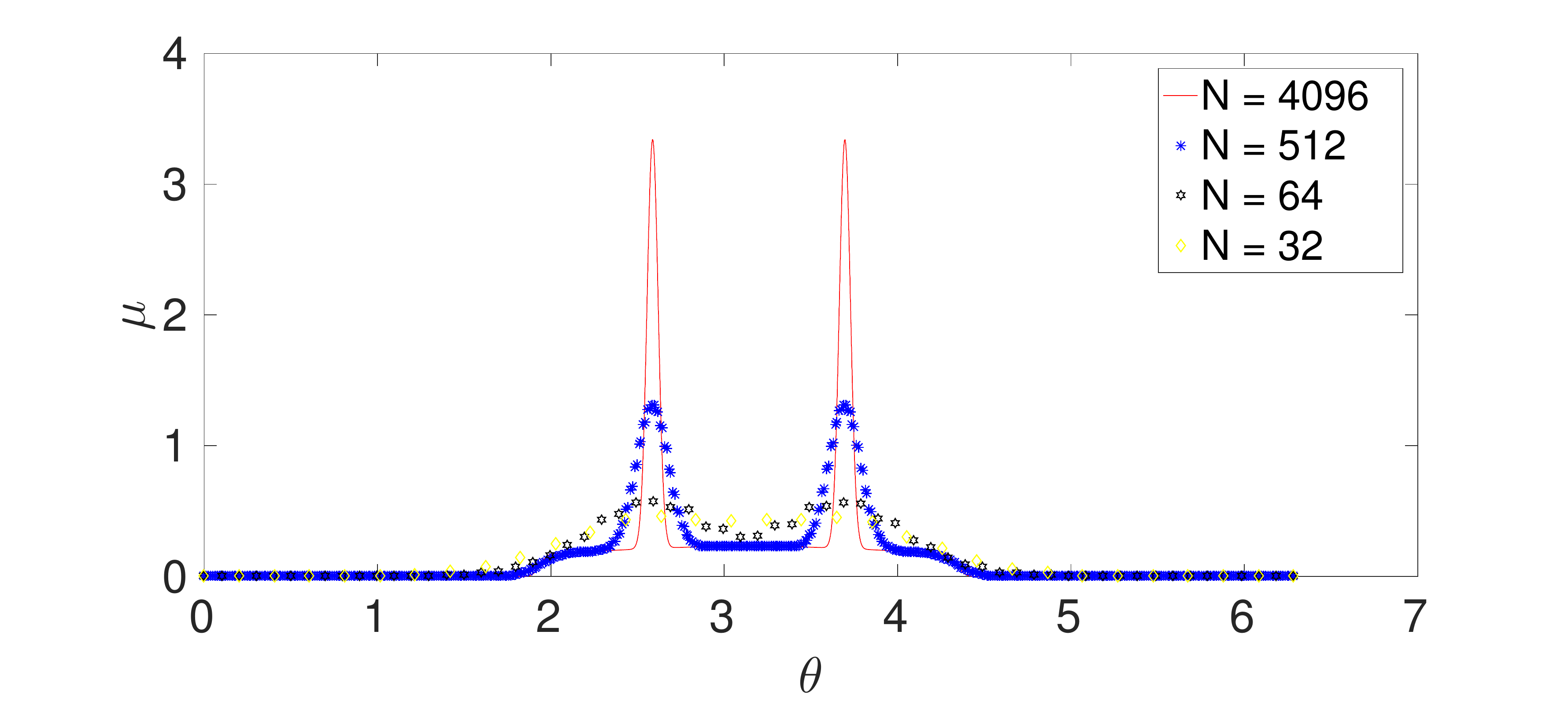}
    \caption{Comparing solutions on different meshes at $t=0.5$ for initial data with Dirac masses.}
    \label{fig:diracinitial}
\end{figure}
\begin{table}
\centering
\begin{tabular}{c|cc|cc}
& \multicolumn{2}{c|}{1-Wasserstein} & \multicolumn{2}{c}{$L^1$} \\
$N$   & Error   & EOC & Error   & EOC \\ \hline
32  & 0.1737 &  & 0.8703                                            &                                                \\ 
64  & 0.1119 & 0.63                  & 0.7945                                            & 0.13                                           \\
128 & 0.0687 & 0.70                  & 0.7131                                            & 0.16                                           \\ 
256 & 0.0433 & 0.67                  & 0.6153                                            & 0.21                                           \\ 
512 & 0.0260 & 0.74                  & 0.4861                                           & 0.34                                           \\
\end{tabular}
\caption{Errors and EOC in the $1$-Wasserstein and $L^1$ distances at $t=0.5$ for singular initial data.}
\label{tab:diracinitial}
\end{table}

\subsection{Polynomial initial data in 2-D}\label{sec:numexp4}
The final numerical experiment approximates the two-dimensional Kuramoto equation \eqref{eq:ks}. We consider the piecewise linear initial data
%
%
\[
 \mu_{0}(\theta,\Omega) = 
    \begin{cases}
      \frac{64}{3\pi^{2}} \theta\Omega & \text{if } \theta\in[\frac{\pi}{4},\frac{\pi}{2}) \text{ and } \Omega \in [0,1] \\
      0 & \text{otherwise}.
 \end{cases}
\]
Figure \ref{fig:2d} shows the solution at $t=0.5$ and Table \ref{tab:2d} shows the $L^1$ errors. (Due to the complexities of computing the Wasserstein distance for multi-dimensional measures, we only compute the $L^1$ errors in this experiment.) The convergence rate in $L^1$ is seen to be about the same as in the piecewise constant one-dimensional example in Section \ref{sec:numexp2}.

\begin{figure}[h]
    \centering
    \includegraphics[width=8cm]{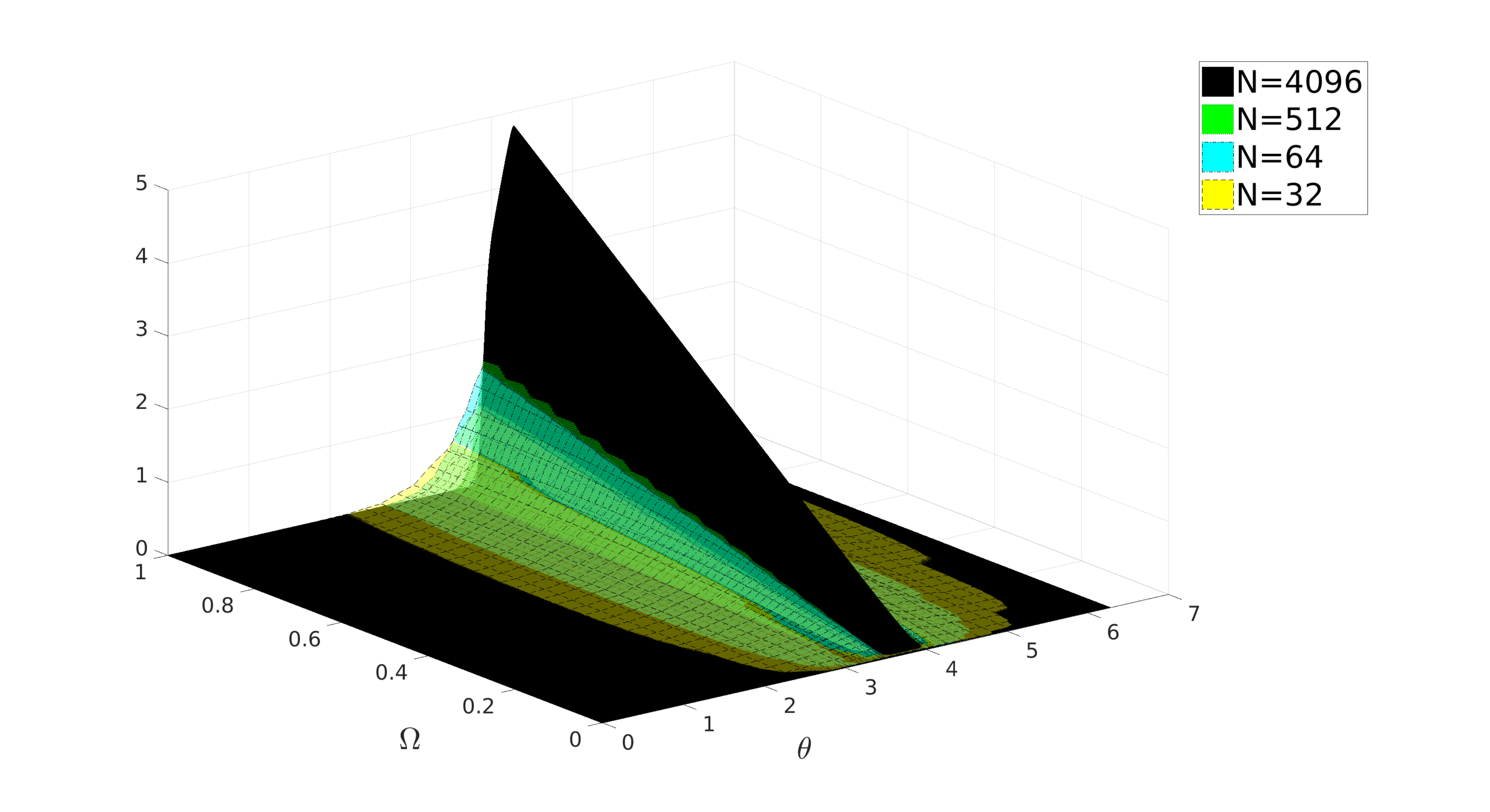}
    \caption{Comparing solutions on different meshes for two-dimensional polynomial initial data.}
    \label{fig:2d}
\end{figure}
\begin{table}[h]
\centering
\begin{tabular}{c|c c}
& \multicolumn{2}{c}{$L^1$} \\
\hline
$N$   & Error   & EOC \\ \hline
32                         & 1.2912                      &      \\
64                         & 1.0054                      & 0.36                      \\
128                        & 0.7346                      & 0.45                      \\
256                        & 0.4925                      & 0.58                      \\
512                        & 0.2999                      & 0.72                      \\
1024 & 0.1638 & 0.87 \\
\end{tabular}
\caption{Table of $L^{1}$-errors for two-dimensional initial data.}
\label{tab:2d}
\end{table}

\appendix
\section{Appendix}
\begin{proof}[Proof of Lemma \ref{lem:compactness}]
We claim that $K$ is tight. Define 
\[
\chi(x):=\begin{cases}
1 & |x|\leq1 \\
2-|x| & 1<|x|<2 \\
0 & 2\leq|x|,
\end{cases}
\qquad \chi_M(x) := \chi(x/M).
\]
For an $\eps>0$, let $M>0$ be such that $\bar{\mu}(\R\setminus B_M(0)) \leq \eps$ and let $C:=\sup_{\mu\in K}d_W(\mu,\bar{\mu})$. If $\mu\in K$ then
\begin{align*}
\mu\big(\R\setminus B_{2M}(0)\big) &\leq \ip{\mu}{1-\chi_M} = \ip{\bar{\mu}}{1-\chi_M} + \ip{\mu-\bar{\mu}}{1-\chi_M} \\
&\leq \bar{\mu}(\R\setminus B_M(0)) + \|1-\chi_M\|_{\Lip}d_W(\mu,\bar{\mu}) \\
&\leq \eps + \frac{C}{M},
\end{align*}
which can be made arbitrarily small, independently of $\mu$. Hence, $K$ is tight, so by Prokhorov's theorem, $K$ is relatively compact in the topology of narrow (or ``weak'') convergence. But narrow convergence is equivalent to $d_W$-convergence. This completes the proof.
\end{proof}

\end{document}